\documentclass[12pt]{amsart}

\usepackage{fancyhdr}
\usepackage{amsmath, amsthm}
\usepackage{indentfirst}
\usepackage{amssymb}
\usepackage{enumerate}
\usepackage[a4paper]{geometry}
\usepackage{dsfont}
\usepackage{graphicx}
\usepackage{lipsum}
\usepackage{hyperref}
\usepackage{color}
\usepackage{paralist}
\usepackage{pgfplots}
\usepackage{tabularx,array}

\newtheorem{theo}{Theorem}[section]

\newtheorem{prop}[theo]{Proposition}

\theoremstyle{definition}
\newtheorem{defn}[theo]{Definition}
\newtheorem*{defn*}{Definition}

\newtheorem{lm}[theo]{Lemma}

\newtheorem{expl}{Example}[section]

\date{} 
\title[Boundedness of PFIOS]{On the boundedness of periodic Fourier integral operators  in Lebesgue spaces with variable exponent}
\vspace{2cm}

\begin{document}

\author[B. Tai]{Boukary Tai}
\address{Boukary Tai\\Department of Mathematics \\ University Norbert ZONGO \\ BP 376 Koudougou \\ Burkina Faso}
\email{taiboukary@gmail.com}

\author[M. Congo]{Mohamed Congo}
\address{Mohamed Congo\\Department of Mathematics  \\ University Joseph Ki-Zerbo \\ 03 BP 7021 Ouagadougou\\ Burkina Faso}
\email{mohamed.congo@yahoo.fr}

\author[M. F. Ouedraogo]{Marie Fran{\c{c}}oise Ouedraogo}
\address{Marie Fran{\c{c}}oise Ouedraogo\\ Department of Mathematics  \\ University Joseph Ki-Zerbo \\ 03 BP 7021 Ouagadougou\\ Burkina Faso}
\email{omfrancoise@yahoo.fr}

\author[A. Ouedraogo]{Arouna Ouedraogo}
\address{Arouna Ouedraogo\\Department of Mathematics  \\ University Norbert ZONGO \\ BP 376 Koudougou \\ Burkina Faso}
\email{arounaoued2002@yahoo.fr}

\maketitle
	
\begin{center}
\textbf{Abstract}
\end{center}
The aim of this paper is to investigate the boundedness of periodic Fourier integral operators in Lebesgue spaces with variable exponent $L^{p(\cdot)}$ on the $n$-dimensional torus. We deal with operators of type $(\rho, \delta)$  which symbols belong to the  H\"{o}rmander  class $S^{m}_{\rho, \delta}(\mathbb{T}^{n}\times\mathbb{Z}^{n})$ for $0\leq\delta<\rho\leq1.$\\

\section{Introduction}

A periodic Fourier integral operator (also called Fourier series operator) is defined by providing a symbol and a phase function. Such operator can be expressed as follows
$$A_{\phi, a}f(x)=\sum_{\xi\in\mathbb{Z}^{n}}e^{2\pi i\phi(x, \xi)}a(x, \xi)(\mathcal{F}_{\mathbb{T}^{n}}f)(\xi),\quad \forall f\in C^{\infty}(\mathbb{T}^{n}),$$ 
where $(\mathcal{F}_{\mathbb{T}^{n}}f)(\xi)$ is the Fourier transform on the torus $\mathbb{T}^{n}$, $a(x,\xi)$  denotes the symbol and $\phi(x, \xi)$ is the phase function. These operators were first introduced by M. Rushansky and V. Turunen  \cite{RT}. They naturally emerged in the solutions of hyperbolic Cauchy problems with periodic conditions, as can be seen, for example in [\cite{RT}, pages 410-411].

The boundedness of  Fourier integral operators in a functional space is contingent upon  conditions on the symbol $a(x,\xi )$ and the phase $\phi(x,\xi ).$ Several authors   have established results on the extension of Fourier integral operators in the $L^{p}(\mathbb{R}^{n})$ spaces  depending on the values of the real order of the symbol (see\cite{Ms} and  \cite{LH}).  Furthermore, D. Ferreira  and W. Staubach \cite{10} investigated the regularity of Fourier integral operators within weighted Lebesgue spaces $L^{p}_w(\mathbb{R}^{n})$, where the weight function belongs to the Muckenhoupt space $A_p,$  for $1 < p < \infty$. In \cite{AI}, K. Alexei Yu and S. Ilya M. studied the boundedness of pseudo differential operators associated to a symbol in certain class of  H\"{o}rmander.

 Also, the study of periodic Fourier integral operators gives rise to a fundamental issue pertaining to a topological property of these operators, namely the question of their boundedness in functional spaces. For example, the study of the $L^{p}$-boundedness of periodic Fourier integral  operators with a symbol $a(x, \xi)\in S^{m}_{1, 0}(\mathbb{T}^{n}\times\mathbb{Z}^{n})$ i.e. $$|\partial_{x}^{\beta}\triangle^{\alpha}_{\xi}a(x, \xi)|\leq C_{\alpha, \beta}\langle \xi  \rangle^{m-|\alpha|},$$ with a positively homogeneous phase function of degree $1$ (for $\xi\neq 0$),  belonging to $C^{\infty}(\mathbb{T}^{n}\times\mathbb{R}^{n}\backslash\{0\})$ whose Hessian matrix is non-degenerate in the spaces $L^{p}(\mathbb{T}^{n})$, was studied by D. Cardona, R. Messiouene and A. Senoussaoui \cite{2}. Moreover, D. Cardona in \cite{Dc}  studied the particular case when the phase function  $\phi(x, \xi)=x\cdot\xi$ and established sufficient conditions on the symbol $a(x,\xi)$ to ensure boundedness of  periodic pseudo-differential operators in the spaces $L^{p}(\mathbb{T}^{n})$.
However, it is obvious that Lebesgue spaces with a constant exponent are not sufficient for modelling complex physical phenomena, in particular those exhibiting spatial variation in properties, such as heterogeneous materials or non-Newtonian fluids. This  led to the generalization to the variable exponent Lebesgue spaces $L^{p(\cdot)}(\mathbb{R}^{n}).$ 
Some authors such as  \cite{00}, \cite{P} focused on the extension of differential operators and pseudo-differential operators in  $L^{p(\cdot)}(\mathbb{R}^{n})$ and $L^{p(\cdot)}_w(\mathbb{R}^{n}).$ 

In this paper, we focus on the study of periodic Fourier integral operators in the Lebesgue spaces with variable exponent on the $n$-dimensional torus $L^{p(\cdot)}(\mathbb{T}^{n}).$ 
 We first establish the boundedness of periodic Fourier integral operators in $L^{p_0}_w(\mathbb{T}^{n}),$  when $1<p_0<\infty$ and  deduce the boundedness of these operators in  $L^{p(\cdot)}(\mathbb{T}^{n}),$ using the technique developed by V. Rabinovich and S. Samko \cite{VS}. We also establish boundedness results for periodic Fourier integral operators in  $L^{p(\cdot)}_w(\mathbb{T}^{n}).$ The rest of the paper is organized as follows. The Section 2 is devoted to preliminaries on variable exponent and weight function in the torus. In section 3, we provide basic tools on periodic Fourier integral operators

\section{Preliminaries}

The two first sections present the basic definitions and useful results in the sequel. For more informations, see references \cite{G.c}, \cite{00}, \cite{D}, \cite{RT}.
\subsection{On the torus}

\begin{enumerate}

\item 	The Torus is the quotient space $$\mathbb{T}^{n}=\mathbb{R}^{n}/\mathbb{Z}^{n}=(\mathbb{R}/\mathbb{Z})^{n},$$ obtained by the equivalence relation $x\sim y\Longleftrightarrow x-y\in\mathbb{Z}^n$, where $\mathbb{Z}^{n}$ denotes the additive group of integral coordinate. 

\item We can identify $\mathbb{T}^{n}$ with the cube $[0, 1)^{n}\subset \mathbb{R}^{n}$, where the measure on the torus coincides with the restriction of the Euclidean measure on the cube.
	
\item 	A function $f:\mathbb{R}^{n}\rightarrow\mathds{C}$ is 1-periodic if  $f(x+l)=f(x)$\quad $\forall x\in\mathbb{R}^{n}$ and $l\in\mathbb{Z}^{n}$.
This definition shows that there is a correspondence between functions defined on $\mathbb{R}^{n}$ and those defined on $\mathbb{T}^{n}$.
\end{enumerate}
To define Fourier integral operators, as well as operator series, we need the notion of Fourier transform.
\begin{defn} The Fourier transform is defined by
$$\mathcal{F}_{\mathbb{T}^{n}}: C^{\infty}(\mathbb{T}^{n})\rightarrow\mathcal{S}(\mathbb{Z}^{n}), \quad
	f\longmapsto\hat{f},$$
where $(\mathcal{F}_{\mathbb{T}^{n}}f)(\xi)=\hat{f}(\xi)={\displaystyle\int_{\mathbb{T}^{n}}}e^{-2\pi i x\cdot\xi}f(x)dx$.\\
\end{defn}
Note that $\mathcal{F}_{\mathbb{T}^{n}}$ is a bijection and its inverse $\mathcal{F}^{-1}_{\mathbb{T}^{n}}: \mathcal{S}(\mathbb{T}^{n})\rightarrow C^{\infty}(\mathbb{Z}^{n})$ is defined by:
$$f(x)=\sum_{\xi \in \mathbb{Z}^n}e^{2\pi i x\cdot\xi}\hat{f}(\xi).$$

\subsection{On periodic Fourier   integral operators}

 Given that the notion of partial derivative is no longer valid when $\xi \in \mathbb{Z}^{n}$, we use the concept of forward and backward difference operators, also known as discrete derivatives.

\begin{defn}
Let $(\delta_j)_{1\leq j\leq n}$  be the canonical basis of $\mathbb{R}^{n}$. For a function $a:\mathbb{Z}^{n}\rightarrow\mathbb{C}$ the forward and backward partial difference operators of $a$ are defined respectively by 
\begin{equation}
	\triangle_{\xi_j}a(\xi)=a(\xi+\delta_j)-a(\xi),\quad \label{eqn6}
\bar{\triangle}_{\xi_j}a(\xi)=a(\xi)-a(\xi+\delta_j).
\end{equation}
For $\alpha=(\alpha_1, \cdots, \alpha_n)\in\mathbb{N}^{n}_{0},$ 
\begin{align}\label{x}
\triangle_{\xi}^{\alpha}=\triangle_{\xi_1}^{\alpha_1}\cdots\triangle_{\xi_n}^{\alpha_n};\qquad
\bar{\triangle}_{\xi}^{\alpha}=\bar{\triangle}_{\xi_1}^{\alpha_1}\cdots \bar{\triangle}_{\xi_n}^{\alpha_n}.
\end{align}
\end{defn}
\begin{lm}[\cite{RT} Lemma 3.3.10]\label{lm}
Assume that $\varphi, \psi:\mathbb{Z}^{n}\rightarrow \mathds{C}$. Then for all $\alpha\in\mathbb{N}^{n}$,
\begin{align}
\sum_{\xi \in \mathbb{Z}^n}\varphi(\xi)\Delta_{\xi}^{\alpha}\psi(\xi)=(-1)^{|\alpha|}\sum_{\xi\in\mathbb{Z}^{n}}\left( \bar{\Delta}^{\alpha}_{\xi}\varphi(\xi)\right) \psi(\xi)
\end{align}
provided that both series are absolutely convergent.
\end{lm}

\begin{defn}
Let $m\in \mathbb{R}, 0 \leq \delta, \rho \leq 1$.
The H\"{o}rmander symbol  class $S^{m}_{\rho,\delta}(\mathbb{T}^n \times\mathbb{Z}^n)$ consists of functions $a(x, \xi)\in C^{\infty}(\mathbb{T}^n \times\mathbb{Z}^n)$ which satisfy the estimation: for $\alpha,\beta\in\mathbb{N}^{n}$,
there exists a constant $C_{\alpha,\beta}> 0$ such that
\begin{align}
\left| \Delta_{\xi}^{\alpha}\partial^\beta_{x}a(x,\xi)\right| \leq C_{\alpha,\beta}\langle\xi\rangle^{m-\rho|\alpha|+\delta|\beta|}, \quad\forall x \in\mathbb{T}^n, \quad\forall\xi\in\mathbb{Z}^n.
\end{align}
The periodic Fourier integral operator (or Fourier series operator) associated to the symbol $a(x,\xi)$  and phase function $\phi(x, \xi)$ denoted by $A_{\phi, a}$ is defined by 
\begin{align}
A_{\phi, a}f(x)=\sum_{\xi\in\mathbb{Z}^{n}}e^{2\pi i\phi(x, \xi)}a(x, \xi)(\mathcal{F}_{\mathbb{T}^{n}}f)(\xi),\quad \forall f\in C^{\infty}(\mathbb{T}^{n}),
\end{align}
where $\phi:\mathbb{T}^{n}\times\mathbb{Z}^{n}\rightarrow \mathbb{R}$  is positively homogeneous  of degree 1 in $\xi\neq0$ and the function $x\mapsto e^{2\pi i\phi(x, \xi)}$ is 1-periodic for all $\xi\in\mathbb{Z}^{n}.$
\end{defn}
In the sequel, the operator $A_{\phi, a}$ will be denoted  $A.$\\
In \cite{RT}, the authors mentioned the result below which  shows that properties of toroidal symbols automatically imply certain properties for differences. The proof follows their Proposition 3.3.4:
	let $a(x, \xi)\in C^{k}(\mathbb{T}^{n}\times\mathbb{Z}^{n})$, $k\in\mathbb{N}$. For every $\alpha\in\mathbb{N}^{n}$, and $\beta\in\mathbb{N}^{n}$, $|\beta|\leq k$ we have the identity 
\begin{align}
	\Delta_{\xi}^{\alpha}\partial_{x}^{\beta}a(x, \xi)=\sum_{|\gamma|\leq|\alpha|}(-1)^{|\alpha-\gamma|}
	\begin{pmatrix}
	\alpha \\ 
	\gamma
	\end{pmatrix} \partial_{x}^{\beta}a(x, \xi+\gamma),\quad \forall (x, \xi)\in\mathbb{T}^{n}\times\mathbb{Z}^{n}.
\end{align}

\subsection{Some basic tools on variable exponent and weight functions }

\begin{defn}
Let $\mathcal{P}(\mathbb{T}^{n})$ be the set of all measurable and 1-periodic functions $p(\cdot):\mathbb{T}^{n}\rightarrow (0,\infty]$ and let  $p_-=ess \inf_{x\in\mathbb{T}^{n}}p(x)$ and $p_+=ess\sup_{x\in\mathbb{T}^{n}}p(x).$
The function $p(\cdot)\in\mathcal{P}(\mathbb{T}^{n})$ is said locally log-H\"older continuous, abbreviated $p \in C^{\log}_{\text{loc}}(\mathbb{T}^{n})$, if there exists a constant $c_{\log}(p) > 0$ such that
$$|p(x)-p(y)|\leq\frac{c_{\log}(p)}{-\log|x-y|},\quad x, y\in\mathbb{T}^{n}, |x-y|\leq\frac{1}{2}. $$
\end{defn}

\begin{defn}
Let $p(\cdot)\in\mathcal{P}(\mathbb{T}^{n})$. The variable exponent Lebesgue space $L^{p(\cdot)}(\mathbb{T}^{n})$ is
 the set of all measurable, 1-periodic functions $f$ on $\mathbb{T}^{n}$ such that $\varrho_{p(\cdot)}\left( \frac{f}{\lambda}\right)  < \infty$ for some $\lambda>0$, equipped with the Luxemburg norm
\begin{equation*}
	\left\|f \right\|_{L^{p(\cdot)}(\mathbb{T}^{n})} =\inf\left\lbrace \lambda>0: \varrho_{p(\cdot)}\left( \frac{f}{\lambda}\right) \leq 1\right\rbrace, 
\end{equation*}
where $\varrho_{p(\cdot)}\left(\dfrac{f}{\lambda} \right)=\displaystyle{\int_{\mathbb{T}^{n}}}\left| \dfrac{f(x)}{\lambda}\right| ^{p(x)}dx$.
\end{defn}

\begin{lm}(Theorem 4.3.12, \cite{D})\label{l4} If $p(\cdot)\in\mathcal{P}(\mathbb{T}^{n})$ with $p_+<\infty$, then $C^{\infty}_{0}(\mathbb{T}^{n})$ is dense in $L^{p(\cdot)}(\mathbb{T}^{n}).$
\end{lm}
The following results are extremely useful. They are known in the literature for the Euclidean space $\mathbb{R}^n.$
Let's denote by $M$ the maximal operator  and  $M^{\#}$ the sharp operator:
\begin{align*}
	Mf(x)=\sup_{r>0}\frac{1}{\left|B(x, r)\right|}\int_{B(x, r)}|f(y)|dx,\\
	M^{\#}(f)(x)=\sup_{r>0}\frac{1}{\left|B(x, r)\right|}\int_{{B(x, r)}}\left| f(y)-f_{B}(x)\right|dy, 
\end{align*} 

\medskip
where $f_{B}(x)={\displaystyle\sup_{r>0}\frac{1}{B(x, r)}\int_{{B(x_0, r)}}|f(y)|dy}$.

\begin{theo}\label{mx}
Let $p(\cdot)\in\mathcal{P}(\mathbb{T}^{n})$. Suppose $1<p_-\leq p_+<\infty$. Then the following properties are equivalent:
\begin{enumerate}
\item The maximal operator $M$ is bounded in $L^{p(\cdot)}(\mathbb{T}^{n})$;
\item The maximal operator $M$ is bounded in $L^{p'(\cdot)}(\mathbb{T}^{n})$, with $\frac{1}{p(\cdot)}+\frac{1}{p'(\cdot)}=1$;
\item There exists $p_0>1$ such that the maximal operator is bounded in $L^{\left( p(\cdot)/p_0\right) }(\mathbb{T}^{n})$.
\end{enumerate}
\end{theo}

Since $\mathbb{T}^{n}$ is identified with the cube $[0; 1)^{n}\subset\mathbb{R}^{n}$, the above theorem is similar to the Theorem 3.35 of \cite{DCU} where $\Omega$ or $\mathbb{R}^n$ is replaced by $\mathbb{T}^n.$

\begin{defn}
An operator $T$ is of weak type $(1, 1)$	if there is a constant $C > 0$ such that for every $\lambda > 0$ we have
\begin{align*}
meas\left\lbrace x\in\mathbb{T}^{n} : |Tu(x)|>\lambda\right\rbrace \leq C\frac{\left\|u \right\|_{L^{1}(\mathbb{T}^{n})}} {\lambda}.
\end{align*}
\end{defn}

\begin{theo}[\cite{P} Theorem 2.1]\label{sk}
Let $T$ be a linear operator associated to a kernel $K$ that satisfies the following conditions
\begin{align}\label{a}
\sup_{|\alpha|=1}\sup_{x, y\in\mathbb{T}^{n}}\left\| y\right\|^{n+1}\left| \partial_{x}^{\alpha}K(x, y)\right|<\infty, \\
\sup_{|\beta|=1}\sup_{x, y\in\mathbb{T}^{n}}\left\| x\right\|^{n+1}\left| \partial_{y}^{\beta}K(x, y)\right|<\infty
\end{align}
and $T$ is of weak type (1, 1). Then for $0 < s < 1$, there exists a constant $C_s>0$ such that
	\begin{align}
	M_s^{\#}\left( Tf\right) (x)\leq C_s Mf(x),\quad \forall f\in C^{\infty}_{0}(\mathbb{T}^{n}).
	\end{align}
\end{theo}

The weight functions in $L^{p}$ spaces are useful in the proof of general and precise  regularity results.
\begin{defn} Let $w\in L^{1}_{loc}(\mathbb{T}^{n})$ a non-negative function. Then  $w$ belongs to the Muckenhoupt  weights space $A_{p_0}$ for $1<p_0<\infty$ if 
\begin{align}\label{w}
[w]_{p_0}:=\sup_{\mathcal{Q}}\left( \frac{1}{|\mathcal{Q}|}\int_{\mathcal{Q}}w(x)dx\right)\left( \frac{1}{|\mathcal{Q}|}\int_{\mathcal{Q}}w(x)^{-\frac{1}{p_0-1}}dx\right)^{p_0-1} <\infty,
\end{align} where  $\mathcal{Q}$ is a cube in $\mathbb{T}^{n}$.\\
By definition  $w\in A_1$ if there exists a constant $C>0$  such that $Mw(x)\leq Cw(x)$ for all $x\in\mathbb{T}^{n}$.
\end{defn}

\begin{expl} [\cite{10} Example 1]
The function $ |x|^{\alpha} $ is an $ A_p $ weighted, for $ 1 < p < \infty $, if and only if $ -n < \alpha < n(p-1) $.
\end{expl}

\begin{lm}[\cite{SG} Property 2]\label{wgt}
Suppose that $w$ is in $A_p$ for some $p\in[1,~ \infty]$ and $0 < \delta < 1$. Then $w$ belongs to $A_q$
where $q=\delta p+1-\delta$. Moreover, $[w^{\delta}]_{p}\leq[w]^{\delta}_{p}.$
\end{lm}

Next, we state the extrapolation theorem of Rubio de Francia \cite{DCU} applied to the torus.
\begin{theo}\label{t1}
Suppose that for $p_0>1$ and $\mathcal{F}$ is a family of pairs non-negative measurable functions such that for all $w\in A_1$
\begin{align}
\int_{\mathbb{T}^{n}} F(x)^{p_0}w(x)dx\leq c_{p_0}\int_{\mathbb{T}^{n}} G(x)^{p_0}w(x)dx, \qquad (F, G)\in\mathcal{F}.
\end{align}
If $p(\cdot)\in\mathcal{P}(\mathbb{T}^{n})$, $p_0\leq p_{-}\leq p_{+}<\infty$ and the maximal operator $M$ is bounded on $L^{\left( \frac{p(\cdot)}{p_0}\right) '}(\mathbb{T}^{n})$, then there exists a  constant $C>0$ such that
\begin{align}
	\left\|F \right\|_{L^{p(\cdot)}(\mathbb{T}^{n})} \leq C \left\|G \right\|_{L^{p(\cdot)}(\mathbb{T}^{n})}, \qquad (F, G)\in\mathcal{F}.
\end{align}	
\end{theo}
\begin{proof}
Since the torus is identified to the cube $[0; 1)^{n}\subset\mathbb{R}^{n},$ we replace $\mathbb{R}^{n}$ by $\mathbb{T}^{n}$ in the Theorem 4.24 \cite{DCU}.
\end{proof}

\begin{defn}
Let $w\in L^{1}_{loc}(\mathbb{T}^{n})$ be a weight.
\begin{enumerate}		 	
\item If $1 < p < \infty,$ $L^{p}_{w}(\mathbb{T}^n)$ is the
space of all functions $f :\mathbb{T}^n\longrightarrow \mathbb{C}$ with finite quasi-norm
$$\left\| f\right\|_{L^{p}_{w}(\mathbb{T}^n)}= \int_{\mathbb{T}^{n}}\left| f(x)\right|^{p}w(x)dx. $$
\item If $p(\cdot)\in \mathcal{P}(\mathbb{T}^{n})$ such that $1< p_{-}\leq p(x)\leq p_{+}<\infty,$ $L^{p(\cdot)}_{w}(\mathbb{T}^n)$ is the space of all functions $f :\mathbb{T}^n\longrightarrow \mathbb{C}$ with finite quasi-norm
$$\left\| f\right\|_{L^{p(\cdot)}_{w}(\mathbb{T}^n)}= \left\|wf\right\|_{L^{p(\cdot)}(\mathbb{T}^n)}. $$
\end{enumerate}
\end{defn}
	
\begin{prop}\label{poids}
 Let $p(\cdot)\in \mathcal{P}(\mathbb{T}^{n})$ such that $1< p_{-}\leq p(x)\leq p_{+}<\infty$ and $0<s<p_-.$ If $w\in L^{1}_{loc}(\mathbb{T}^{n})$ is a weight, then
\begin{align}\left\| f\right\|_{L^{p(\cdot)}_{w}(\mathbb{T}^n)}=\left\| f^s\right\|^{\frac{1}{s}}_{L^{\frac{p(\cdot)}{s}}_{w^s}(\mathbb{T}^n)}.
\end{align}
\end{prop}

The following theorem is proved for constant $p$ in the non-weighted case in [\cite{Ste}, p. 148] and for variable $p(\cdot)$ in the weighted case in \cite{02}, Lemma 4.1.
\begin{theo}\label{1t3}Let $T$ be an operator with kernel $K$ such that 
\begin{align*}\label{seri}
	Tf(x)=\int_{\mathbb{T}^{n}}K(x, x-y)f(y)dy.
\end{align*} 
	Let $p(\cdot)\in C^{\log}_{loc}(\mathbb{T}^{n})$ such that $1< p_{-}<p_{+}<\infty$ and  $p(x)=p_{\infty}$ for $|x|\geq R$ where $R>0$. Suppose also a weight function $w\in A_{p(\cdot)}$ of the form
	\begin{align*}
	w(x) = (1 + |x|)^{\beta}\prod_{k=1}^{n}|x - x_k|^{\beta_k}, \quad x_k \in \mathbb{T}^{n}.
	\end{align*}
	Then if~~ 

	$-\dfrac{n}{p(x_k)} < \beta_k < \dfrac{n}{p'(x_k)}\quad and 
\quad	-\dfrac{n}{p_{\infty}} < \beta + \sum_{k=1}^{n}\beta_k < \dfrac{n}{p'_{\infty}}, \quad k = 1, \cdots, n,$
there exists a constant $C>0$ such that
	\begin{align}
	\left\| Tf\right\|_{L^{p(\cdot)}_{w}(\mathbb{T}^{n})} \leq C \left\|M^{\#}(|Tf|) \right\| _{L^{p(\cdot)}_{w}(\mathbb{T}^{n})},\quad \forall f\in C^{\infty}_{0}(\mathbb{T}^{n}).
	\end{align}
\end{theo}

\section{Mains results}

We begin by proving the regularity  of periodic Fourier integral operator in the weighted Lebesgue space $L^{p_0}_w(\mathbb{T}^{n})$, where the weight function is locally integrable and positive. To this purpose let's set up the following lemma.

\begin{lm}\label{lr}
	For all multi-indices $\alpha\in\mathbb{N}^{n} $ and $a(x, \xi)\in C^{\infty}(\mathbb{T}^{n}\times\mathbb{Z}^{n})$ we have\\
	\begin{align*}
		\sum_{\xi\in\mathbb{Z}^{n}}e^{2\pi i (x-y)\cdot\xi}a(x, \xi)=(-1)^{|\alpha|}(e^{2\pi i (x-y)}-1)^{-\alpha}\sum_{\xi\in\mathbb{Z}^{n}}\left(\bar{\triangle}_{\xi}^{\alpha} e^{2\pi i (x-y)\cdot\xi}\right) a(x, \xi).
	\end{align*}
\end{lm}
\begin{proof}
	By using the identity (\ref{eqn6})  we obtain:
	\begin{align*}
		\bar{\Delta}^{\alpha_1}_{\xi_1}e^{2\pi i (x-y)\cdot\xi}&=e^{2\pi i (x-y)\cdot\xi}-e^{2\pi i (x-y)\cdot(\xi+\delta_1)}\\
		&=e^{2\pi i (x-y)\cdot\xi}-e^{2\pi i (x-y)\cdot\xi}\cdot e^{2\pi i(x-y)\cdot\delta_1}\\
		&=-e^{2\pi i (x-y)\cdot\xi}\left( e^{2\pi i (x-y)\cdot\delta_1}-1\right).
	\end{align*}  
	\begin{align*}
		-\bar{\Delta}^{\alpha_1}_{\xi_1}e^{2\pi i (x-y)\cdot\xi}
		&=e^{2\pi i (x-y)\cdot\xi}\left( e^{2\pi i (x-y)\cdot\delta_1}-1\right)\\
		&= e^{2\pi i (x-y)\cdot\xi}\left( e^{2\pi i (x-y)}-1\right).
	\end{align*} 
Now the identity (\ref{x}) with $\alpha=(\alpha_1\cdots\alpha_n)$ and  $\xi=(\xi_1\cdots\xi_n)$ gives
	\begin{align*}
		&(-1)^{|\alpha|}\bar{\Delta}^{\alpha}_{\xi}e^{2\pi i(x-y)\cdot\xi}=\left( e^{2\pi i (x-y)}-1\right)^{\alpha}\cdot e^{2\pi i(x-y)\cdot\xi}.
	\end{align*}
	Thus
	\begin{align*}
		e^{2\pi i(x-y)\cdot\xi}=(-1)^{|\alpha|}\left( e^{2\pi i (x-y)}-1\right)^{-\alpha}\bar{\Delta}^{\alpha}_{\xi}e^{2\pi i(x-y)\cdot\xi}.
	\end{align*}
	This yields the result:
	\begin{align*}
		\sum_{\xi \in \mathbb{Z}^n} e^{2\pi i(x-y)\cdot\xi}a(x, \xi)=(-1)^{|\alpha|}\left( e^{2\pi i (x-y)}-1\right)^{-\alpha}\sum_{\xi \in \mathbb{Z}^n}\left( \bar{\Delta}^{\alpha}_{\xi}e^{2\pi i(x-y)\cdot\xi}\right) a(x, \xi).
	\end{align*}
\end{proof}

\begin{theo}\label{t2}
Let $1< p_0<\infty$ and $w\in A_{p_0}$. Let $A:C^{\infty}(\mathbb{T}^{n})\rightarrow C^{\infty}(\mathbb{T}^{n})$ be the periodic Fourier integral operator defined by 
\begin{align*}\label{01}
		Af(x)=\sum_{\xi\in\mathbb{Z}^{n}}e^{2\pi i\phi(x, \xi)}a(x, \xi)(\mathcal{F}_{\mathbb{T}^{n}}f)(\xi),\quad \forall f\in C^{\infty}(\mathbb{T}^{n}),
\end{align*} 
where $\phi(x, \xi):\mathbb{T}^{n}\times\mathbb{Z}^{n}\rightarrow\mathbb{R}$ is a phase function such that  $x \mapsto e^{2\pi i\phi(x, \xi)}$ is 1-perodic for all $\xi\in\mathbb{Z}^{n}$ and $a(x, \xi): \mathbb{T}^{n} \times \mathbb{Z}^{n} \rightarrow \mathbb{C}$ is a symbol satisfying the H\"{o}rmander condition
\begin{align*}
		|\partial_{x}^{\beta}\triangle^{\alpha}_{\xi}a(x, \xi)|\leq C_{\alpha, \beta}\langle \xi  \rangle^{m-\rho|\alpha|+\delta|\beta|},
\end{align*}
with $m\leq(\rho-1)|\frac{1}{p_0}-\frac{1}{2}|-\epsilon$ and $\alpha, \beta \in\mathbb{N}^{n}$; $\epsilon>\delta$.
Then the periodic Fourier integral operator $A$ is bounded on $L^{p_0}_w(\mathbb{T}^{n}).$
	
\end{theo}

\begin{proof} 
	
The symbol $a(x, \xi) \in S^{m}_{\rho, \delta}(\mathbb{T}^{n} \times \mathbb{Z}^{n})$ is continuous and thus admits a Fourier series expansion  given by $$a(x, \xi)=\sum_{\eta\in\mathbb{Z}^{n}}e^{2\pi ix\cdot\eta} \hat{a}(\eta, \xi)~~\text{for all}~\eta \in \mathbb{Z}^{n}.$$
 Let us assume that $f\in C^{\infty}_{0}(\mathbb{T}^{n})$. The decomposition of the phase function $\phi(x, \xi)=x\cdot\xi+\psi(\xi)$, where $\psi(\xi)$ is a real values function belonging to $C^{\infty}(\mathbb{R}^{n}\backslash\left\lbrace 0\right\rbrace )$ and is positively homogeneous of degree $1$ in $\xi\neq 0$ gives
\begin{align}
	Af(x)&=\sum_{\xi\in\mathbb{Z}^{n}}e^{2\pi i\phi(x, \xi)}a(x, \xi)(\mathcal{F}_{\mathbb{T}^{n}}f)(\xi)\nonumber\\
	&=\sum_{\xi\in\mathbb{Z}^{n}}e^{2\pi i(x\cdot\xi+\psi(\xi))}a(x, \xi)(\mathcal{F}_{\mathbb{T}^{n}}f)(\xi)\nonumber\\
	&=\sum_{\xi\in\mathbb{Z}^{n}}\sum_{\eta\in\mathbb{Z}^{n}}e^{2\pi i(x\cdot\xi+\psi(\xi))}e^{2\pi ix\cdot\eta} \hat{a}(\eta, \xi)(\mathcal{F}_{\mathbb{T}^{n}}f)(\xi)\nonumber\\
	&=\sum_{\eta\in\mathbb{Z}^{n}}e^{2\pi ix\cdot\eta}\left( \sum_{\xi\in\mathbb{Z}^{n}}e^{2\pi i x\cdot\xi} \hat{a}(\eta, \xi)(\mathcal{F}_{\mathbb{T}^{n}}f)(\xi)e^{2\pi i\psi(\xi)}\right). \nonumber
\end{align}
One can see the expression 
\begin{align}
	 \sum_{\xi\in\mathbb{Z}^{n}}e^{2\pi i x\cdot\xi} \hat{a}(\eta, \xi)(\mathcal{F}_{\mathbb{T}^{n}}f)(\xi)e^{2\pi i\psi(\xi)}
\nonumber
\end{align}
as the symbol of the product of the two operators $\hat{a}(\eta, D_x)$ and $e^{2\pi i\psi(D_x)}.$ Namely
	\begin{align*}
	\left( \hat{a}(\eta, D_x)e^{2\pi i\psi(D_x)}\right)f(x) 
		=\sum_{\xi\in\mathbb{Z}^{n}}e^{2\pi ix\cdot\xi}  \hat{a}(\eta, \xi)e^{2\pi i\psi(\xi)}(\mathcal{F}_{\mathbb{T}^{n}}f)(\xi) .
	\end{align*}
	It follows that 
	$$Af(x)=\sum_{\eta\in\mathbb{Z}^{n}}e^{2\pi ix\cdot\eta}  \hat{a}(\eta, D_x)f(x)e^{2\pi i\psi(D_x)} ,
	$$  
where $$\hat{a}(\eta, D_x)f(x)=\sum_{\xi\in\mathbb{Z}^{n}}e^{2\pi ix\cdot\xi}\hat{a}(\eta, \xi)(\mathcal{F}_{\mathbb{T}^{n}}f)(\xi)$$ is the Fourier multiplier. \\
We now estimate $Af$ with respect to $\hat{a}(\eta, D_x)f$.

\begin{eqnarray}\label{003}
\left\| Af\right\|_{L^{p_0}_w(\mathbb{T}^{n})}&\leq &\left\lbrace \int_{\mathbb{T}^{n}}\sum_{\eta\in\mathbb{Z}^{n}}\left|e^{2\pi i x\cdot\xi} \hat{a}(\eta, D_x)f(x)e^{2\pi i\psi(D_x)}\right|^{p_0}w(x)dx\right\rbrace ^\frac{1}{p_{0}}\nonumber\\
&\leq &\sum_{\eta\in\mathbb{Z}^{n}}\left\lbrace\int_{\mathbb{T}^{n}}\left| \hat{a}(\eta, D_x)f(x)\right|^{p_0}w(x)dx\right\rbrace^\frac{1}{p_{0}} \nonumber\\
&\leq &\sum_{\eta\in\mathbb{Z}^{n}}\left\|  \hat{a}(\eta, D_x)f(x)\right\| _{L^{p_0}_w(\mathbb{T}^{n})}.
\end{eqnarray}
The next step is to estimate the norm $\left\|  \hat{a}(\eta, D_x)f\right\| _{L^{p_0}_w(\mathbb{T}^{n})}$. 
\begin{align}
\left\|  \hat{a}(\eta, D_x)f\right\| _{L^{p_0}_w(\mathbb{T}^{n})}\nonumber
&=\left\lbrace \int_{\mathbb{T}^{n}}\left| \hat{a}(\eta, D_x)f(x)\right| ^{p_0}w(x)dx\right\rbrace ^{\frac{1}{p_0}}\\ \nonumber
&\leq\left\lbrace \int_{\mathbb{T}^{n}}\left| \sum_{\xi\in\mathbb{Z}^{n}}e^{2\pi i x\cdot\xi}\hat{a}(\eta, \xi)\hat{f}(\xi)\right|^{p_0}w(x)dx\right\rbrace ^{\frac{1}{p_0}}\\\nonumber
&\leq\left\lbrace \int_{\mathbb{T}^{n}}\left| \sum_{\xi\in\mathbb{Z}^{n}}\int_{\mathbb{T}^{n}}e^{2\pi i (x-y)\cdot\xi}\hat{a}(\eta, \xi)f(y)dy\right|^{p_0}w(x)dx\right\rbrace ^{\frac{1}{p_0}}\\ \nonumber
&=\left\lbrace \int_{\mathbb{T}^{n}}\left|\int_{\mathbb{T}^{n}} \sum_{\xi\in\mathbb{Z}^{n}}e^{2\pi i (x-y)\cdot\xi}\hat{a}(\eta, \xi)f(y)dy\right|^{p_0}w(x)dx\right\rbrace ^{\frac{1}{p_0}}.\\ \nonumber
\end{align}

We then apply Lemma \ref{lr} to deduce:
\begin{eqnarray*}
&{}& \left\|  \hat{a}(\eta, D_x)f\right\| _{L^{p_0}_w(\mathbb{T}^{n})} \\ 
&\leq &\left\lbrace \int_{\mathbb{T}^{n}}\left| \int_{\mathbb{T}^{n}} \sum_{\xi\in\mathbb{Z}^{n}}\left[ \l(-1)^{|\alpha|}(e^{2\pi i(y-x)}-1)^{-\alpha}\bar{\triangle}^{\alpha}_\xi e^{2\pi i (x-y)\cdot\xi}{a}(\eta, \xi)\right] f(y)dy\right|^{p_0}w(x)dx\right\rbrace ^{\frac{1}{p_0}}.
\end{eqnarray*}
By Lemma \ref{lm}, and the second-order multidimensional Taylor expansion around $(0, 0)$ to the function $e^{2\pi i (x-y)}$:

\begin{align*}
e^{2 \pi i(x-y)}&=1+\nabla e^{0}\cdot 2\pi i(x-y)+o(\left\| x-y\right\|^2)\\
&=1+ 2\pi i(x-y) + o(\left\| x-y\right\|^2),
\end{align*}
where $\nabla e^{0}=(1, -1)$ is the gradient of the function $e^{2\pi i(x-y)}$ at the point $(0, 0)$, we obtain

\begin{eqnarray*}
&{}&\left\|  \hat{a}(\eta, D_x)f\right\| _{L^{p_0}_w(\mathbb{T}^{n})}\\
&\leq &\left\lbrace \int_{\mathbb{T}^{n}}\left|\int_{\mathbb{T}^{n}} \sum_{\xi\in\mathbb{Z}^{n}}\left[ (2\pi i)^{-|\alpha|}\left(   x-y \right)  ^{-\alpha} e^{2\pi i (x-y)\cdot\xi}\triangle ^{\alpha}_\xi\hat{a}(\eta, \xi)\right] f(y)dy\right|^{p_0}w(x)dx\right\rbrace ^{\frac{1}{p_0}}\\ 
&\leq &\left\lbrace \int_{\mathbb{T}^{n}}\left|  \sum_{\xi\in\mathbb{Z}^{n}}\int_{\mathbb{T}^{n}}\left| (2\pi i)^{-|\alpha|}\left(   x-y \right)  ^{-\alpha}f(y)\right| dy |\triangle^{\alpha}_\xi\hat{a}(\eta, \xi)| \right|^{p_0} w(x)dx\right\rbrace ^{\frac{1}{p_0}}\\ 
&\leq & (2\pi)^{-|\alpha|}\left\lbrace \int_{\mathbb{T}^{n}}\left| \sum_{\xi\in\mathbb{Z}^{n}}\left| \triangle^{\alpha}_\xi\hat{a}(\eta, \xi)\right| \int_{\mathbb{T}^{n}}|\left(  x-y \right) ^{-\alpha}f(y)|dy\right|^{p_0} w(x)dx\right\rbrace ^{\frac{1}{p_0}} \\
&\leq & (2\pi)^{-|\alpha|}\left\lbrace \int_{\mathbb{T}^{n}}\left|  \sum_{\xi\in\mathbb{Z}^{n}}\left| \triangle^{\alpha}_\xi\hat{a}(\eta, \xi)\right| \left( |\left(   x-\cdot \right)  ^{-\alpha}|\star |f(\cdot)|\right) (x) \right| ^{p_0}  w(x)dx\right\rbrace ^{\frac{1}{p_0}} \\
&\leq & (2\pi)^{-|\alpha|} \left( \sum_{\xi\in\mathbb{Z}^{n}}\left| \triangle^{\alpha}_\xi\hat{a}(\eta, \xi)\right|^{p_0}\right) ^{\frac{1}{p_0}} \left\lbrace \int_{\mathbb{T}^{n}}\left|   |\left(x-\cdot \right)  ^{-\alpha}|\star |f(\cdot)|(x)\right|^{p_0} w(x)dx\right\rbrace ^{\frac{1}{p_0}} .
\end{eqnarray*}

Note that the convolution norm on the weighted spaces gives
\begin{align*}
	\left\lbrace \int_{\mathbb{T}^{n}}\left|\left|  (x-\cdot)  ^{-\alpha}\right|\star |f(\cdot)|(x)\right|^{p_0} w(x)dx\right\rbrace ^{\frac{1}{p_0}}
	&\leq \left\| \left| \cdot \right| ^{-|\alpha|}\right\| _{L^{1}_w(\mathbb{T}^{n})}\left\| f\right\|_{L^{p_0}_w(\mathbb{T}^{n})}\\
	&\leq C_{\alpha}\left\| f\right\|_{L^{p_0}_w(\mathbb{T}^{n})}.
\end{align*} 

Let's use the estimate 
\begin{align}
\left| \triangle^{\alpha}_\xi\hat{a}(\eta, \xi)\right|\leq C_{r, \alpha}\langle \eta \rangle^{-r}\langle\xi \rangle^{m-\rho|\alpha|+r\delta}, \forall r\in\mathbb{N}_{0}
\end{align}
established in \cite{RT}, Lemma 4.2.1. We obtain
\begin{align*}
\left\| \hat{a}(\eta, D_x)f\right\| _{L^{p_0}_w(\mathbb{T}^{n})}&\leq (2\pi)^{-|\alpha|}C_{\alpha} \left( \sum_{\xi\in\mathbb{Z}^{n}}C_{r, \alpha}\langle \eta \rangle^{-rp_0}\langle\xi \rangle^{(m-\rho|\alpha|+r\delta)p_0}\right)^{\frac{1}{p_0}} \left\| f\right\|_{L^{p_0}_w(\mathbb{T}^{n})} \nonumber\\
&\leq(2\pi)^{-|\alpha|}C_{r, \alpha}\langle \eta \rangle^{-r} \left( \sum_{\xi\in\mathbb{Z}^{n}}\langle\xi \rangle^{(m-\rho|\alpha|+r\delta)p_0}\right) ^{\frac{1}{p_0}} \left\| f\right\|_{L^{p_0}_w(\mathbb{T}^{n})}.
\end{align*}
The series converge for $1<r\leq [\frac{\epsilon}{\delta}]+1$.
So, there exists a constant $C_{p_0,r, \alpha}>0$ such that
\begin{align*}
	\left\| \hat{a}(\eta, D_x)f\right\| _{L^{p_0}_w(\mathbb{T}^{n})}&\leq (2\pi)^{-|\alpha|}C_{r, \alpha}\langle \eta \rangle^{-r}C_{p_0,r, \alpha}\left\| f\right\|_{L^{p_0}_w(\mathbb{T}^{n})}\\
	&\leq C'_{p_0,r, \alpha} \langle \eta \rangle^{-r}\left\| f\right\|_{L^{p_0}_w(\mathbb{T}^{n})}.
\end{align*}
We are now ready to formulate the boundedness of the operator $A.$ If we go back to the estimate (\ref{003}) we can write  
\begin{align*}
\left\| Af\right\|_{L^{p_0}_w(\mathbb{T}^{n})}&\leq\sum_{\eta\in\mathbb{Z}^{n}}\left\|  \hat{a}(\eta, D_x)f(x)\right\| _{L^{p_0}_w(\mathbb{T}^{n})}\\&
\leq\sum_{\eta\in\mathbb{Z}^{n}} C'_{p_0, r, \alpha}\langle \eta \rangle^{-r} \left\| f\right\|_{L^{p_0}_w(\mathbb{T}^{n})}\\\nonumber
&\leq \left( C'_{p_0, r, \alpha}\sum_{\eta\in\mathbb{Z}^{n}}\langle \eta \rangle^{-r} \right)  \left\| f\right\|_{L^{p_0}_w(\mathbb{T}^{n})}.
\end{align*}
Since $r>1$, the sum $\sum_{\eta\in\mathbb{Z}^{n}}\langle \eta \rangle^{-r} $ is finite and there exists a constant $C''_{p_0, r, \alpha}>0$ such that
$$\left\| Af\right\|_{L^{p_0}_w(\mathbb{T}^{n})}\leq C''_{p_0, r, \alpha}\left\|f\right\|_{L^{p_0}_w(\mathbb{T}^{n})}.$$
\end{proof}

Now, we present a sufficient condition for the boundedness of periodic Fourier integral operators in generalized Lebesgue spaces $L^{p(\cdot)}(\mathbb{T}^{n})$, using Rubio de Francia extrapolation theorem on the torus.
\begin{theo} \label{t3}
Let  $\phi(x, \xi)\in C^{\infty}(\mathbb{T}^{n}\times\mathbb{Z}^{n}  )$  be a  phase function such that $x\mapsto e^{2\pi i\phi(x, \xi)}$ is 1-periodic for all $\xi\in\mathbb{Z}^{n}$ and  let  $a(x, \xi)\in C^{\infty}(\mathbb{T}^{n}\times\mathbb{Z}^{n})$ be a symbol which satisfies
 \begin{equation*}
 |\triangle^{\alpha}_{\xi}\partial_{x}^{\beta}a(x, \xi)|\leq C_{\alpha, \beta}\langle \xi \rangle^{m-\rho|\alpha|} \quad \forall \alpha, \beta\in\mathbb{N}^{n}.
 \end{equation*}
For all $p(\cdot)\in \mathcal{P}(\mathbb{T}^{n})$ such that  $1<p_-\leq p(\cdot)\leq p_+<\infty$,
 there exists a constant $C>0$ such that
 the periodic Fourier integral operator $A$ associated to the symbol $a(x, \xi)$ satisfies 
	\begin{align*}
	\left\|  Af\right\| _{L^{p(\cdot)}(\mathbb{T}^{n})} \leq C\left\|  f \right\| _{L^{p(\cdot)}(\mathbb{T}^{n})}.
	\end{align*}	
\end{theo}

\medskip

\begin{proof}
	Let $f \in C^{\infty}_{0}(\mathbb{T}^{n})$. Then $f \in L^{p(\cdot)}(\mathbb{T}^{n})$ since $C^{\infty}_{0}(\mathbb{T}^{n})$ is dense in $L^{p(\cdot)}(\mathbb{T}^{n})$ (Lemma \ref{l4}) .\\
	Moreover, if $1 \leq p < q < \infty$,  there is a continuous embedding of Muckenhoupt classes $A_{p} \hookrightarrow A_{q}$ (see Lemma \ref{wgt}). Let $w$ be a  weight function in $A_1$.
	Since $p_0 > 1$, then   $w \in A_{p_0}$.
		By Theorem \ref{t2}, the operator $A$ is bounded in $L^{p_0}_{w}(\mathbb{T}^{n})$.
	Moreover, the maximal operator is bounded on $L^{(p(\cdot)/p_0)'}(\mathbb{T}^{n}) $ and $(|Af|, |f|)$ is a pair of positive functions.
	By Theorem \ref{t1}, the periodic Fourier integral operator $A$ is bounded in $L^{p(\cdot)}(\mathbb{T}^{n})$, and there exists a constant $C> 0$ such that
$$\left\|  Af\right\| _{L^{p(\cdot)}(\mathbb{T}^{n})} \leq C\left\|  f \right\| _{L^{p(\cdot)}(\mathbb{T}^{n})}.
$$
\end{proof}
Hereafter is the result of boundedness for periodic Fourier integral operators in weighted Lebesgue spaces with variable exponent. 

\begin{theo}\label{t5}
	Let $a(x, \xi): \mathbb{T}^{n} \times \mathbb{Z}^{n} \rightarrow \mathbb{C}$ be a symbol which satisfies the condition
	\begin{align*}
	|\partial_{x}^{\beta}\Delta_{\xi}^{\alpha}a(x, \xi)| \leq C_{\alpha, \beta}\langle\xi\rangle^{m - \rho|\alpha| + \delta|\beta|},
	\end{align*}
	where $0 \leq \delta < \rho \leq 1$, the parameter $m <-(n+1) $ for all $\alpha, \beta\in \mathbb{N}^{n}_{0}$.
Let $\phi: \mathbb{T}^{n} \times \mathbb{Z}^{n} \rightarrow \mathbb{R}$ be a phase function such that the function $x\mapsto e^{2\pi i\phi(x, \xi)}$ is 1-periodic   and satisfies the condition: there exist a constant $C>0$ such that
	\begin{align*}
	|\partial^{\alpha}_{x}\phi(x, \xi)| \leq C.
	\end{align*}
Further, suppose $p(\cdot) \in C^{\log}_{loc}(\mathbb{T}^{n})$ such that $1< p_{-}\leq p_{+}<\infty$ and  $p(\cdot)=p_{\infty}$ for $|x|\geq R$ where $R>0$. Let $w \in A_{p(\cdot)}$ be a Muckenhoupt weight function of the form $\displaystyle {w(x) = (1 + |x|)^{\beta}\prod_{k=1}^{n}|x - x_k|^{\beta_k}}$ such that for all $x_k \in \mathbb{T}^{n}$,
	\begin{align}\label{wi}
	-\dfrac{n}{p(x_k)} < \beta_k < \dfrac{n}{p'(x_k)}, and 
\quad	-\dfrac{n}{p_{\infty}} < \beta + \sum_{k=1}^{m}\beta_k < \dfrac{n}{p'_{\infty}}, \quad k = 1, \dots, n.
	\end{align}
Then the periodic Fourier integral operator $A$ associated to the symbol $a$  is  bounded in $L^{p(\cdot)}_{w}(\mathbb{T}^{n})$, and there exists a  constant $C' > 0$ such that  
\begin{align*}
\left\| Af\right\| _{L^{p(\cdot)}_{w}(\mathbb{T}^{n})}\leq C'\left\| f \right\| _{L^{p(\cdot)}_{w}(\mathbb{T}^{n})}.
\end{align*}               
\end{theo}


\begin{proof} To simplify, we use the expression $f\lesssim g$ which means that $f \leq cg$ for some independent	constant $c>0$.\\
Let $f\in C^{\infty}_{0}(\mathbb{T}^{n})$ and  $0<t<1$. We have
 \begin{align*}
\left\| Af\right\| _{L^{p(\cdot)}_{w}(\mathbb{T}^{n})}&=
 \left\| \left| Af\right|^{t}  \right\|^{\frac{1}{t}} _{L^{\frac{p(\cdot)}{t}}_{w^t}(\mathbb{T}^{n})}.
\end{align*}
In what follows, we will decompose the proof in two steps.\\ 
\textbf{First step:} In this step, we establish the estimate
\begin{align*}
\left\|Af \right\|_{L^{p(\cdot)}_{w}(\mathbb{T}^{n})}\lesssim \left\| M^{\#}|Af|^{t}\right\|^{\frac{1}{t}}_{L^{\frac{p(\cdot)}{t} }_{w^{t}}(\mathbb{T}^{n})}.
\end{align*} 

\noindent Let $p(\cdot)\in C^{\log}_{loc}(\mathbb{T}^{n})$. For all $0<t<1$, we have  $\frac{p(\cdot)}{t}\in C^{\log}_{loc}(\mathbb{T}^{n})$. Multiplying inequality (\ref{wi}) by $t$ gives 
\begin{align*}
-\dfrac{n}{\frac{p(x_k)}{t}}<t\beta_k<\dfrac{n}{\frac{p'(x_k)}{t}}\quad and 
\quad	-\dfrac{n}{\frac{p_{\infty}}{t}} < t\beta + \sum_{k=1}^{n}t\beta_k < \dfrac{n}{\frac{p'_{\infty}}{t}}, \quad k=1, ..., n.
\end{align*} 
Furthermore, for all $0<t<1$, we have $tp_-+1-t<p_-$. Thus by Lemma \ref{wgt}, the weight function $w^{t}$ belongs to $A_{p_{-}}.$ Since the space $A_{p_{-}}\subset A_{p(\cdot)}$, then 
$w^{t}\in A_{p(\cdot)}$ .
Now, using Theorem \ref{1t3} and the density of $C^{\infty}_{0}(\mathbb{T}^{n})$  in $L^{p(\cdot)}_{w}(\mathbb{T}^{n})$ (see Lemma \ref{l4}), for all $f\in C^{\infty}_{0}(\mathbb{T}^{n})$, 
\begin{align*}
\left\| Af\right\|_{L^{p(\cdot)}_{w}(\mathbb{T}^{n})}\lesssim & \left\| M^{\#}\left( |Af|^{t}\right) \right\|^{\frac{1}{t}}_{L^{\frac{p(\cdot)}{t} }_{w^{t}}(\mathbb{T}^{n})}.
\end{align*}
\noindent\textbf{Second step:} For the second step we show that
\begin{align*}
\left\| M^{\#}\left( |Af|^{t}\right) \right\|^{\frac{1}{t}}_{L^{\frac{p(\cdot)}{t} }_{w^{t}}(\mathbb{T}^{n})}\lesssim \left\| f \right\|_{L^{p(\cdot)}_{w}(\mathbb{T}^{n})}.
\end{align*}
Let's establish the  conditions of the Theorem \ref{sk}, where $K(x,y)$ is the kernel associated to the  periodic Fourier integral operator $A:$
\begin{align*}
\sup_{|\alpha'|=1}\sup_{x, y\in\mathbb{T}^{n}}\left\| y\right\|^{n+1}\left| \partial_{x}^{\alpha'}K(x, y)\right|<\infty, \\ 
\sup_{|\beta'|=1}\sup_{x, y\in\mathbb{T}^{n}}\left\| x\right\|^{n+1}\left| \partial_{y}^{\beta'}K(x, y)\right|<\infty .
\end{align*}
By using the Leibniz formula as well as the estimates on the symbol $a$ and the phase $\phi$, we obtain 
\begin{align*}
\partial_{x}^{\alpha'}K(x, y)=&\partial_{x}^{\alpha'}\left( \sum_{\xi\in\mathbb{Z}^{n}} e^{2\pi i(\phi(x, \xi)-y\cdot\xi)}a(x, \xi) \right)\\
=& \sum_{\xi\in\mathbb{Z}^{n}}\sum_{|\gamma|\leq|\alpha'|}C_{\alpha, \gamma}(2\pi i)^{|\gamma|}\partial_{x}^{\gamma}\phi(x, \xi)\partial_{x}^{\alpha'-\gamma}a(x, \xi)e^{2\pi i(\phi(x, \xi)-y\cdot\xi)}.
\end{align*}
Consequently
\begin{align*}
\left| \partial_{x}^{\alpha'}K(x, y)\right| \leq& \sum_{\xi\in\mathbb{Z}^{n}}\sum_{|\gamma|\leq|\alpha'|}\left| (2\pi i)^{|\gamma|}\partial_{x}^{\gamma}\phi(x, \xi)\partial_{x}^{\alpha'-\gamma}a(x, \xi)e^{2\pi i(\phi(x, \xi)-y\cdot\xi)} \right| \\
\leq&\sum_{\xi\in\mathbb{Z}^{n}}\sum_{|\gamma|\leq|\alpha'|}C_{\alpha', \gamma}\left| \partial_{x}^{\gamma}\phi(x, \xi) \right| \left|\partial_{x}^{\alpha'-\gamma}a(x, \xi) \right|\\
\leq&\sum_{\xi\in\mathbb{Z}^{n}}\sum_{|\gamma|\leq|\alpha'|}C'_{\alpha', \gamma}C\langle\xi \rangle^{m+\delta|\alpha'-\gamma|}\ <\infty.
\end{align*} Since $m +\delta |\alpha' - \gamma| < -n$, this ensures convergence with respect to $\xi$. Moreover, by multiplying $\left| \partial_{x}^{\alpha'}K(x, y)\right|$ by $\left| y\right|^{n+1}$ and considering the identification of $\mathbb{T}^{n}$ with the cube $[0, 1)^{n}$, such that for all $ y \in \mathbb{T}^{n}$ we have $\left| y\right| \leq 1$, we can conclude that 
$\displaystyle{\sup_{|\alpha'|=1}\sup_{x, y\in\mathbb{T}^{n}}\left\| x\right\|^{n+1}\left| \partial_{y}^{\alpha'}K(x, y)\right|<\infty .}$\\

We now give an estimate of $|\partial_{y}^{\beta'} K(x, y)|$.
	\begin{align*}
\left| 	\partial^{\beta'}_{y}K(x, y)\right| \leq&\sum_{\xi\in\mathbb{Z}^{n}}|(2\pi i)^{|\beta|}||\xi|^{|\beta|}|a(x, \xi)|\\
\leq& \sum_{\xi\in\mathbb{Z}^{n}}C'\langle\xi \rangle^{m+|\beta'|}.
	\end{align*} 
	Using the precedent idea for $x\in\mathbb{T}^{n}$ we obtain
	$\displaystyle{\sup_{|\beta'|=1}\sup_{x, y\in\mathbb{T}^{n}}\left\| x\right\|^{n+1}\left| \partial_{x}^{\beta'}K(x, y)\right|<\infty .}$\\
Note also that if the symbol of integral operators is order $m<-(n+1)$ then this operator is a locally weak (1, 1) [Seeger \cite{SSS}].
Since the kernel satisfy (\ref{a}) and $A$ is a locally weak $(1, 1),$  the Theorem \ref{sk} yields
\begin{align*}
\left\| M^{\#}\left( |Af|^{t}\right) \right\|^{\frac{1}{t}}_{L^{\frac{p(\cdot)}{t} }_{w^{t}}(\mathbb{T}^{n})}\lesssim\left\| M\left( |f|^{t}\right) \right\|^{\frac{1}{t}}_{L^{\frac{p(\cdot)}{t} }_{w^{t}}(\mathbb{T}^{n})}
=\left\| M\left( f\right) \right\|_{L^{p(\cdot)} _{w}(\mathbb{T}^{n})}.
\end{align*}  
Moreover, the maximal operator $M$ is bounded in $L_{w}^{p(\cdot)}(\mathbb{T}^{n})$  (Theorem \ref{mx}).  Thus, by the density of $C^{\infty}_{0}(\mathbb{T}^{n})$ in $L^{p(\cdot)}(\mathbb{T}^{n})$ (Lemma \ref{l4}), we have
\begin{align*}
\left\| M^{\#}\left( |Af|^{t}\right) \right\|^{\frac{1}{t}}_{L^{\frac{p(\cdot)}{t} }_{w^{t}}(\mathbb{T}^{n})}\lesssim&\left\| M\left( f\right) \right\|_{L^{p(x)} _{w}(\mathbb{T}^{n})}\\
\lesssim&\left\| f\right\| _{L^{p(\cdot)}_{w}(\mathbb{T}^{n})}.
\end{align*} 
It follows that    
\begin{align*}
\left\| Af\right\| _{L^{p(\cdot)}_{w}(\mathbb{T}^{n})}\lesssim\left\| f\right\| _{L^{p(\cdot)}_{w}(\mathbb{T}^{n})}.
\end{align*}
\end{proof}


\begin{thebibliography}{25}
	
	
\bibitem{P} J. Alvarez and C. P\'erez. 
Estimates with a weighted for various singular integral operators. Boll. Un. Mat. Ital. A (7), 8(1): p. 123-133, 1994.

\bibitem{00} I. Aydin, R. Akg\"un
weighted variable exponent grand Lebesgue spaces and inequalities of approximation. Hacet. J. Math. Stat. Volume 50 (1) (2021), p. 199-215 DOI : 10.15672/hujms.683997.

\bibitem{Dc} D. Cardona, 
On the boundedness of periodic pseudo-differential operators, Monat. Math, 185(2), p. 189-206, (2017).

\bibitem{2} D. Cardona, R. Messiouene and A. Senoussaoui, 
Lp-bounds for Fourier integral operators on the torus. Complex Variables And Elliptic Equations, 69, p. 252-269, (2024).

\bibitem{DCU} D. Cruz-Uribe and A. Fiorenza,
\emph{Variable Lebesgue spaces: Foundations and harmonic analysis}. Springer Science and Business Media, 2013.

\bibitem{D} L. Diening, P. Harjulehto, P. H\"ast\"o, and M. Rozicka,
\emph{Lebesgue and Sobolev Spaces with Variable exponent.} Lecture Notes in Mathematics 2017, 2011.


\bibitem{G.c} C. Garetto, 
Lp and Sobolev boundedness of pseudodifferential operators with non-regular symbol: A regularisation approach. J. Math. Anal. Appl. 381, p. 328-343 (2011).

\bibitem{SG} S. Ghimire, 
Some Properties of $A_p$ weight Function. Journal of the Institute of Engineering, 2016, 12(1): p. 210-213.

\bibitem{LH} L. H\"ormander, 
Fourier integral operators. I, Acta Math, 127(1-2), p. 79-183, (1971).



\bibitem{02} V. Kokilashvili and S. Samko. Singular Integrals in Weighted Lebesgue Spaces with
Variable Exponent. Georgian Math. J., 10(1): p. 145-156, 2003.


\bibitem{VS} V. Rabinovich and S. Samko, 
Boundedness and Fredholmness of pseudodifferential operators in variable exponent spaces. Integral Equations And Operator Theory, 60 (2008), p. 507-537.

\bibitem{SW} S. Rodriguez-L\'opez and W. Staubach, 
Estimates for rough Fourier integral and pseudo-differential operators and applications to the boundedness of multilinear operators. J. Functional Analysis, 264(10), (2013), p. 2356-2385.

\bibitem{RT} M. Ruzhansky and V. Turunen, 
\emph{Pseudo-Differential Operators and Symmetries: Background Analysis and Advanced Topics.} Birkh\"auser, Basel, 2010.

\bibitem{Ms} M. Ruzhansky and M. Sugimoto,
A local-to-global boundedness argument and Fourier integral operators. Journal of Mathematical Analysis and Applications, 2019, vol. 473, no.2, p. 892-904.

\bibitem{10} D. D. Santos Ferreira and W. Staubach,
\emph{Global and local regularity of Fourier integral operators on weighted and unweighted spaces.} American Mathematical Society, 2014.


\bibitem{SSS} A. Seeger, C. D. Sogge, E.M. Stein, Regularity properties of Fourier integral operators, Ann.
of Math. (2) 134 (1991), no. 2, p. 231-251.


\bibitem{Ste} E. M. Stein, 
\emph{Harmonic Analysis. Real variable methods, orthogonality and oscillatory integrals.} Princeton Univ. Press, Princeton, 1993.

\bibitem{AI} K. A. Yu and S. Ilya M. 
Pseudodifferential operators on variable Lebesgue spaces. In : Operator Theory, Pseudo-Differential Equations, and Mathematical Physics: The Vladimir Rabinovich Anniversary Volume. Basel : Springer Basel, 2012. p. 173-183.




\end{thebibliography}
\end{document}